\documentclass[12pt]{amsart} 
\usepackage{amsmath,amscd} 
\usepackage{latexsym,amsmath,amssymb,mathrsfs,setspace,enumerate}
\usepackage{amssymb}
\usepackage[utf8]{inputenc}
\usepackage[english]{babel}

\theoremstyle{plain}

\newtheorem{lemma}{Lemma} 
\newtheorem{prop}{Proposition} 
\newtheorem{thm}{Theorem}
\newtheorem{exam}{Example}
\newtheorem{cor}{Corollary}

\newtheorem{defn}{Definition}

\author {P. G. Romeo and Arjun. S. N}
\address{Dept. of Mathematics, Cochin University of Science and Technology, 
	Kochi, Kerala, INDIA.}
\email{$romeo_-parackal@yahoo.com,\, arjunsnmaths1996@gmail.com $}
\thanks{} 
\date{}
\begin{document}
	\title{On category of Lie algebras}
	\maketitle
	
	\begin{abstract}
		In this paper we describe the the category of  Lie algebras of group algebras and the category of Plesken Lie algebras and explore the catgorical relations between them.   Further we provide the examples of the Lie algebra of the group algebra of subgroups of Heisenburg group and the Plesken Lie algebra of subgroups of Heisenburg group.
		
	\end{abstract}
	\maketitle

	\section{INTRODUCTION}
	The concept of Lie algebras was introduced by Sophus Lie to solve problems in Lie groups with some ease.  Here we introduce some class of Lie algebras like Lie algebras of group algebras and Plesken Lie algebras of  groups. 	In recent times category theory established itself as a practical tool in dealing with mathematical structures. In particular a categorical approach enables to extract more insight into the interactions between Lie groups and Lie algebras. Here we discuss the category of some class of Lie algebras, functorial relation that exists  between them with some examples.
	\hspace{1.5in}
	\section{PRELIMINARIES} 	 
	In the following we briefly recall all basic definitions and 
	the elementary concepts needed in the sequel. In particular we recall the 
	definitions of  Lie algebra, group algebra, Plesken Lie algebra,  categories, functors and  discusses some interesting properties of these structures. 
	
	\begin{defn}(cf.\cite{kss}) A category ${\mathcal C} $ consists of the following data$:$ 
		
		\begin{enumerate}
			\item A class called the class of vertices or objects $\nu {\mathcal C}. $ 
			\item A class of disjoint sets ${\mathcal C}(a,b)$ one for each pair $ (a,b) 
			\in 
			\nu {\mathcal C}\times \nu {\mathcal C}.$  An element $f \in{\mathcal C}$ is 
			called a morphism from $ a $ to $ b, $ written $ f : a \rightarrow b$  ; $ a = 
			dom\; f$ called the domain of $f$  and $ b = cod\; f $  called the codomain of 
			$ 
			f . $  
			\item For $ a, b, c, \in \nu{\mathcal C}, $ a map
			\newline
			$ \circ: {\mathcal C} ( a,b) \times {\mathcal C}( b , c)  \rightarrow  
			{\mathcal 
				C}( a , c) \hspace{40 pt} ( f , g )  \rightarrow g \circ f $
			\newline
			$\circ $  is called the $ composition $ of morphisms in $ {\mathcal C} . $
			\item  for each $ a \in \nu {\mathcal C} $, a  unique $  1_a \in {\mathcal C}( 
			a,a ) $ is  called the identity morphism on $a.$
		\end{enumerate}
		
		These must satisfy the following axioms :
		\begin{itemize}
			\item(cat1)   The composition is associative : for $ f \in  {\mathcal C} ( a,b) 
			, g \in  {\mathcal C}( b , c)\; and\; h \in   {\mathcal C}( c , d ), $ we have 
			\newline
			$$ f \circ ( g \circ h) = (f \circ g) \circ h $$
			\item (cat 2) for each $ a \in \nu {\mathcal C}, f \in {\mathcal C}\; (a,b) \; 
			and\; g \in {\mathcal C} (c , a),$
			\\ \begin{center}  $ 1_a \circ f = f \qquad   and \qquad g \circ 1_a = g  
				$\end{center}.
		\end{itemize}
	\end{defn}
	
	Clearly $\nu \mathcal C$ can be identify as a subclass of $\mathcal C$ and with 
	this identification
	it is possible to regard categories in terms of morphisms alone. The category 
	$\mathcal C$ is said to
	be small if the class $\mathcal C$ is a set. A morphism $f\in \mathcal C(a, b)$ 
	is said to be an
	isomorphism if there exists $f^{-1} \in \mathcal C(b, a)$ such that $ff^{-1} = 
	1_a = e_a,$ domain identity
	and $f^{-1}f = 1_b= f_b,$ range identity. 
	
	\begin{exam}
		A group $G$ can be regarded as a category $ \mathcal{C}$ with  the object set of $\mathcal{C}$ say $\nu C={G}$, and morphisms  $\mathcal C(G, G) = G$ and composition in $\mathcal C$ is the binary 
		operation 
		in $G.$ Identity element in the group will be the identity morphism on  the 
		vertex $G.$
	\end{exam}
	\begin{defn}\cite{maclane}
		For categories $ \mathcal{C}$ and $\mathcal{D}$ a functor $T :\mathcal{C} \rightarrow \mathcal{D}$ with domain $\mathcal{C}$ and codomain $\mathcal{D}$ consists of two functions: the object function $T$, which assigns to each object $c$ of $\mathcal{C}$ an object $Tc$ of $\mathcal{D}$ and the arrow function which assigns to each arrow $f : c \rightarrow c'$  of $\mathcal{C}$ an arrow $Tf : Tc \rightarrow Tc'$ of $\mathcal{D}$, in such a way that $T(1_{c}) = 1_{Tc}$ and $T(g \circ f) = Tg \circ Tf$. 
	\end{defn}
	\begin{defn}\cite{Humphreys}
		A vector space $L$ over a field $\mathbb{F}$, with an operation $L \times L \rightarrow L$, denoted by $(x, y) \mapsto [x,y]$ for $x$ and $y$ in $L$  and satisfying  the following axioms: 
		\begin{enumerate}
			\item The bracket operation is bilinear.  For $x,y,z \in L, a,b \in \mathbb{F}$
			\begin{center}
				$[ax+by, z]= a[x, z]+b[y, z]$\\
				$[x, ay+bz] = a[x,y]+b[x,z]$
			\end{center}
			\item $[x,x] =0$ for all $x \in L$
			\item Jacobi identitiy :\\$[x,[y,z]]+[y,[z,x]]+[z,[x,y]]=0$ for all $x,y,z \in L$     
		\end{enumerate} 
		is called a bracket product and $(L, [.,.])$ is called a Lie algebra over $\mathbb{F}$. 
	\end{defn}
	\begin{exam}
		$End(V)$, the set of all linear transformations on a finite 
		dimensional vector space V over a field $\mathbb{F}$  is a Lie algebra with Lie bracket $[x, y] = xy - yx$ for $x, y \in End(V)$.
	\end{exam}
	A subspace $K$ of a Lie algebra $L$ is called a \textit{Lie subalgebra} if $[x, y] \in K$ whenever $x, y \in K$. 
	\begin{defn}
		Let $G$ be a group and let $\mathbb{F}$ be $\mathbb R$ or $\mathbb C$. Define a vector space over $\mathbb{F}$ with elements of $G$ as a basis, and denote it by $\mathbb{F}G$.That is; $\mathbb{F}G=\{ \displaystyle\sum_{i} a_ig_i: a_i \in \mathbb{F} \ \text{for all i} \}$.The addition and scalar multiplication in $\mathbb{F}G$ are defined by; for
		\begin{center}
			$\alpha=\displaystyle\sum_{i=0} a_ig_i  \ \text{and} \  \beta=\displaystyle\sum_{i=0} b_ig_i $
		\end{center}
		in $\mathbb{F}G$ and $k \in \mathbb{F}$, 
		\begin{center}
			${\alpha}+\beta=\displaystyle\sum_{i}(a_i+b_i)g_i$ \text{and}  
			$ k \alpha=\displaystyle\sum_{i}(k a_i)g_i$
		\end{center}
		The vector space $\mathbb{F}G$, with multiplication defined by 
		\begin{center}
			$\begin{pmatrix} \displaystyle\sum_{i} a_ig_i \end{pmatrix}
			\begin{pmatrix} \displaystyle\sum_{j}b_jg_j \end{pmatrix}=\displaystyle\sum_{i,j }a_i b_j(g_ig_j) , \ \indent a_i,b_j \in \mathbb{F} $
		\end{center} 
		is called the group algebra of $G$ over $\mathbb{F}$.
	\end{defn}
	The group algebra of a finite group $G$ is a vector space of dimension $|G|$ which also
	carries extra structure involving the the product operation on $G$.
	\begin{exam}
		$G=C_3= \ <a:a^3=e>$ where e is the identity in $G$. Then 
		\begin{center}
			$\mathbb{C}G=\{ \lambda_1e+\lambda_2a+\lambda_3a^2 : \lambda_i \in \mathbb C \  \text{for} \  i=1,2,3 \}$
		\end{center} 
		$\mathbb{C}G$ is a group algebra with usual addition and multiplication of series.
	\end{exam}
	
	\section{Lie algebras of group algebras and Plesken Lie algebras}
	\indent Let $\mathbb{F}G$ be a group algebra over $\mathbb{F}$. Then $\mathbb{F}G$ can be regarded as a Lie algebra by defining the Lie bracket $[\ , \ ] : \mathbb{F}G \times \mathbb{F}G \rightarrow \mathbb{F}G$ as follows : 
	\begin{center}
		$[ \alpha  ,  \beta   ] = \alpha\beta  - \beta \alpha$ for $\alpha=\displaystyle\sum_{i} a_ig_i  \ \text{and} \  \beta=\displaystyle\sum_{i} b_ig_i $ in $\mathbb{F}G$
	\end{center} 
	Clearly $\mathbb{F}G$ is a Lie algebra with respect to the given Lie bracket and is
	called the Lie algebra of the group algebra $\mathbb{F}G$ and we denote it by $L_{\mathbb{F}G}$.\\
	A linear map between $L_{\mathbb{F}G}$ and $L_{\mathbb{F}H}$ which preserves the Lie bracket  is a \textit{homomorphism} of Lie algebras of group algebras. 
	
	\begin{prop}
		Let $f : G \rightarrow H$ be a group homomorphism.
		Then $\bar{f} : L_{\mathbb{F}G} \rightarrow L_{\mathbb{F}H} $ defined by 
		\begin{center}
			$\bar{f}(\alpha) = \bar{f}(\displaystyle\sum_{i} a_ig_i)=\displaystyle\sum_{i} a_if(g_i)$ 
		\end{center}			
		is a homomorphism between Lie algebras of group algebras.
	\end{prop}
	\begin{proof}
		Let $G= \{g_1, g_2, g_3 . . . ,  \}$ and $L_{\mathbb{F}G} = 	\{ \displaystyle\sum_{i} a_ig_i: a_i \in \mathbb{F} \ \text{for all i} \}$. Then for  $\alpha = \displaystyle\sum_{i} a_ig_i, \beta = \displaystyle\sum_{j} b_jg_j \in L_{\mathbb{F}G}$,
		
		\begin{equation*}
			\begin{split}
				\bar{f}([\alpha, \beta])&=	\bar{f}(\alpha\beta - \beta\alpha)\\
				&=\bar{f}(\displaystyle\sum_{i} a_ig_i\displaystyle\sum_{j} b_jg_j - \displaystyle\sum_{j} b_jg_j\displaystyle\sum_{i} a_ig_i) \\
				&= \bar{f}(\displaystyle\sum_{i,j} a_ib_jg_ig_j - \displaystyle\sum_{j,i} b_ja_ig_jg_i)\\
				&=\displaystyle\sum_{i,j} a_ib_jf(g_ig_j) - \displaystyle\sum_{j,i} b_ja_if(g_jg_i)\\
				&= \displaystyle\sum_{i,j} a_ib_jf(g_i)f(g_j) - \displaystyle\sum_{j,i} b_ja_if(g_j)f(g_i)\\
				&= \displaystyle\sum_{i} a_if(g_i)\displaystyle\sum_{j} b_jf(g_j) - \displaystyle\sum_{j} b_jf(g_j)\displaystyle\sum_{i} a_if(g_i) \\
				&= [\bar{f}(\alpha), \bar{f}(\beta)]
			\end{split}		
		\end{equation*}	
		Hence, $\bar{f}$ is a homomorphism between Lie algebras of group algebras.
	\end{proof}	
	\ \indent Next we proceed to describe Plesken Lie algebra. Let $G$ be a group and $\mathbb{F}G$ be its group algebra over $\mathbb{F}$, then for each $g \in G$, $g - g^{-1} \in \mathbb{F}G$, denote it by $\hat{g}$, then the linear span of $\hat{g}$  admits a Lie algebra structure as expailned below.
	\begin{defn}\cite{Plesken}
		\ \indent  Plesken Lie algebra $\mathcal{L}(G)$ of a  group $G$ over $\mathbb{F}$   is the linear span of elements $\hat{g} \in \mathbb{F}G$ together with the Lie bracket
		\begin{center}
			$[\hat{g}, \hat{h}] = \hat{g}\hat{h} - \hat{h}\hat{g}$
		\end{center}
	\end{defn}
	That is, for any group $G = \{g_1, g_2, g_3, . . .  \}$,  $ \{\displaystyle\sum_{i} {a_i}\hat{g_i}: a_i \in \mathbb{F} \ \text{for all i} \}$ together with the Lie bracket defined above is the Plesken Lie algebra $\mathcal{L}(G)$.
	\begin{lemma}
		The Plesken Lie algebra $\mathcal{L}(G)$ over $\mathbb{F}$ is a Lie subalgebra of the Lie algebra  $L_{\mathbb{F}G}$.
	\end{lemma}
	\begin{proof}
		Let $G= \{g_1, g_2, g_3, . . . \}$ be a group. Then the Lie algebra of the group algebra $\mathbb{F}G$ is $L_{\mathbb{F}G} = 	\{ \displaystyle\sum_{i} {\lambda_i}{g_i}: a_i \in \mathbb{F} \ \text{for all i} \}$ and the Plesken Lie algebra is  $\mathcal{L}(G) = \{\displaystyle\sum_{i} {a_i}\hat{g_i}: a_i \in \mathbb{F} \ \text{for all i} \}$.
		Since $\mathcal{L}(G)$ is the linear span of $\hat{g}$ and $ \hat{g} \in L_{\mathbb{F}G}$,  $\mathcal{L}(G)$ is a subset of $L_{\mathbb{F}G}$.\\ Let $\hat{\alpha}=\displaystyle\sum_{i} {a_i}\hat{g_i}  \ \text{and} \  \hat{\beta}=\displaystyle\sum_{i} {b_i}\hat{g_i} $ in $\mathcal{L}(G)$,\\
		\begin{equation*}
			\begin{split}
				\hat{\alpha}+\hat{\beta}&=\displaystyle\sum_{i} {a_i}\hat{g_i}+\displaystyle\sum_{i} {b_i}\hat{g_i}\\
				&=\displaystyle\sum_{i}({a_i}+{b_i})\hat{g_i} \in \mathcal{L}(G)		
			\end{split}
		\end{equation*}
		and $k(\hat{\alpha}) =k \displaystyle\sum_{i} {a_i}\hat{g_i}= \displaystyle\sum_{i} (k{a_i})\hat{g_i} \in \mathcal{L}(G)$. \\
		Thus, $\mathcal{L}(G)$ is a subspace of $L_{\mathbb{F}G}$.\\
		Let $\hat{g}, \hat{h} \in \mathcal{L}(G)$, then
		\begin{equation*}
			\begin{split}
				[\hat{g}, \hat{h}] &= \hat{g}\hat{h} - \hat{h}\hat{g} \\
				&= (g - g^{-1})(h - h^{-1}) - (h - h^{-1})(g - g^{-1}) \\
				&=  \widehat{gh}- \widehat{gh^{-1}}- \widehat{g^{-1}h}+ \widehat{g^{-1}h^{-1}}
			\end{split}
		\end{equation*}
		Thus Lie bracket is closed in $\mathcal{L}(G)$. Hence, $\mathcal{L}(G)$ is a Lie subalgebra of $L_{\mathbb{F}G}$.
	\end{proof}	
	\begin{exam} 
		Consider the symmetric group $S_3$, then  
		\begin{equation*}
			\begin{split}
				L(S_3) &= \text{span} \{ \sigma - {\sigma}^{-1} : \sigma \in S_3 \}\\
				&= \{ a_1((1) - (1)) + a_2((1 \ 2) - (1 \ 2)) +a_3((1 \ 3) - (1 \ 3)) + +a_4((2 \ 3) - (2 \ 3)) \\
				& \ \indent \quad + a_5((1 \ 2 \ 3) - (1 \ 3 \ 2))+ a_6((1 \ 3 \ 2) - (1 \ 2 \ 3)) : a_i \in \mathbb{C} \}\\
				&= \{ a((1 \ 2 \ 3) - (1 \ 3 \ 2)) : a \in \mathbb{C} \}
			\end{split}
		\end{equation*} 
		is a one dimensional  Plesken Lie algebra over $\mathbb{C}$ with Lie bracket
		\begin{center}
			$[a\widehat{(1 \ 2 \ 3)}, b \widehat{(1 \ 2 \ 3)}] = 0$
		\end{center}
	\end{exam}
	A linear map between two Plesken Lie algebras $\mathcal{L}(G)$ and $\mathcal{L}(H)$ is a \textit{Plesken Lie algebra homomorphism} if it preserves the Lie bracket.
	
	\begin{prop}
		Let $f : G \rightarrow H$ be a group homomorphism. Then $\hat{f} : \mathcal{L}(G) \rightarrow \mathcal{L}(H)$ defined by
		\begin{center}
			$\hat{f}(\displaystyle \sum_i a_i \hat{g_i}) = \displaystyle \sum_i a_i \widehat{f(g_i)}$
		\end{center}
		is a Plesken Lie algebra homomorphism.
	\end{prop}
	\begin{proof}
		For $\displaystyle \sum_i a_i \hat{g_i}, \displaystyle \sum_i b_i \hat{g_i} \in \mathcal{L}(G)$ and $k \in \mathbb{F}$,
		\begin{equation*}
			\begin{split}
				\hat{f}(\displaystyle \sum_i a_i \hat{g_i} + k\displaystyle \sum_i b_i \hat{g_i}) &= \hat{f}(\displaystyle \sum_i (a_i + kb_i) \hat{g_i})\\
				&= \displaystyle \sum_i (a_i + kb_i) \widehat{f(g_i)}\\
				&=\hat{f}(\displaystyle \sum_i a_i \hat{g_i}) + k \hat{f}(\displaystyle \sum_i b_i \hat{g_i})
			\end{split}
		\end{equation*}
		thus $\hat{f}$ is linear.
		
		For $\displaystyle \sum_i a_i \hat{g_i}, \displaystyle \sum_j b_j \hat{g_j}  \in \mathcal{L}(G)$,
		\begin{equation*}
			\begin{split}
				\hat{f}([\displaystyle \sum_i a_i \hat{g_i}, \displaystyle \sum_j b_j \hat{g_j}])&= \hat{f} (\displaystyle \sum_{i,j} (a_ib_j - b_ia_j)\widehat{g_ig_j}- \displaystyle \sum_{i,j}(a_ib_j - b_ia_j)\widehat{g_i{g_j}^{-1}} \\
				& \hspace{0.15in}- \displaystyle \sum_{i,j}(a_ib_j - b_ia_j)\widehat{{g_i}^{-1}g_j} + \displaystyle \sum_{i,j}(a_ib_j - b_ia_j)\widehat{{g_i}^{-1}{g_j}^{-1}})\\
				&= \displaystyle \sum_{i,j} (a_ib_j - b_ia_j)\widehat{f(g_i)f(g_j)}- \displaystyle \sum_{i,j}(a_ib_j - b_ia_j)\widehat{f(g_i)f({g_j}^{-1})} \\
				& \hspace{0.15in}- \displaystyle \sum_{i,j}(a_ib_j - b_ia_j)\widehat{f({g_i}^{-1})f(g_j)} + \displaystyle \sum_{i,j}(a_ib_j - b_ia_j)\widehat{f({g_i}^{-1})f({g_j}^{-1})}\\
				&=[\displaystyle \sum_i a_i \widehat{f(g_i)}, \displaystyle \sum_j a_j \widehat{f(g_j)}]
			\end{split}
		\end{equation*}
	 Hence $\hat{f}$ preserves Lie bracket and is a Plesken Lie algebra homomorphism.
	\end{proof}
	
	\section{CATEGORY OF LIE ALGEBRAS}
Here  we describe the category of Lie algebras over $\mathbb{F}$ whose objects are Lie algebras over $\mathbb{F}$ and morphisms Lie algebra homomorphisms. Let $L$ and $L'$ be two Lie algebras over a field $\mathbb{F}$. If $f$ and $g$ are two morphisms, then $f \circ g$ exists only when $f \in hom(L, L')$ and $g \in hom(L', L'')$. Also the identity in $hom(L,L)$ is the morphism $1_L$.  \\

	\begin{exam}
		Consider a  group $G$ and all its subgroups $H_i$. The group algebras $\mathbb{F}H_i$  of each subgroup $H_i$ of $G$ together with a  Lie bracket $[x, y] = xy - yx$ for $x= \displaystyle \sum_i a_ih_i, y = \displaystyle \sum_j b_jh_j \in \mathbb{F}H_i$ is the Lie algebra $L_{\mathbb{F}H_i}$. The collection of all such Lie algebras of group algebras of subgroups of $G$ form the category $L_{\mathbb{F}G}$ whose morphisms are \{ $\bar{f}_{ij} : L_{\mathbb{F}H_i} \rightarrow L_{\mathbb{F}H_j} \vert \bar{f}(\displaystyle \sum_i a_ig_i) = \displaystyle \sum_i a_if(g_i)$ where $f : H_i \rightarrow H_j$ is the group homomorphism  \}.
	\end{exam}	
	\begin{exam}	
		Consider the  Heisenberg group $H(\mathbb{R}) = \{ \begin{pmatrix}1&a&b\\0&1&c\\0&0&1\end{pmatrix} : a, b, c \in \mathbb{R} \}$ and its non-isomorphic subgroups :
		\begin{equation*}
			\begin{split}
				H_1&=\{I_{3 \times3}, \text{ \ the idenity matrix \ }\},
				H_2= \{ \begin{pmatrix}1&0&b\\0&1&0\\0&0&1\end{pmatrix} : b \in \mathbb{R}\}\\
				H_3&=\{\begin{pmatrix}1&0&b\\0&1&c\\0&0&1\end{pmatrix} : b, c \in \mathbb{R}\}, \			H_4= H(\mathbb{R})
			\end{split}
		\end{equation*}
		\ \indent The group  algebras $\mathbb{F}H_i$ over $\mathbb{F}$ of each subgroups $H_i$ of $H(\mathbb{R})$ together with a Lie bracket $[X, Y] = XY - YX$ where $X = \displaystyle \sum_i \lambda_iA_i, Y = \displaystyle \sum_i \mu_iB_i \in \mathbb{F}H_i$ is the Lie algebra $L_{\mathbb{F}H_i}$ of the group algebra $\mathbb{F}H_i$. Consider  $L_{\mathbb{F}H(\mathbb{R})}$ whose objects are
		\begin{equation*}
			\begin{split}
				L_{\mathbb{F}H_1} &= \{ \lambda I : \lambda \in \mathbb{F} \} = \text{ \  the Lie algebra of scalar matrices \  }\\
				L_{\mathbb{F}H_2}&= \{ \displaystyle \sum_i \lambda_i A_i : A_i \in H_2, \lambda_i \in \mathbb{F} \}= \{ \begin{pmatrix}\displaystyle \sum_i \lambda_i&0&\displaystyle \sum_i \lambda_ib_i\\0&\displaystyle \sum_i \lambda_i&0\\0&0&\displaystyle \sum_i \lambda_i\end{pmatrix} \}\\
				L_{\mathbb{F}H_3}&= \{ \displaystyle \sum_i \lambda_i A_i : A_i \in H_3, \lambda_i \in \mathbb{F} \}
				= \{ \begin{pmatrix}\displaystyle \sum_i \lambda_i&0&\displaystyle \sum_i \lambda_ib_i\\0&\displaystyle \sum_i \lambda_i&\displaystyle \sum_i \lambda_ic_i\\0&0&\displaystyle \sum_i \lambda_i\end{pmatrix} \}\\
			\end{split}
		\end{equation*}
		\begin{equation*}
			\begin{split}
				L_{\mathbb{F}H_4}&= \{ \displaystyle \sum_i \lambda_i A_i : A_i \in H_4, \lambda_i \in \mathbb{F} \}
				= \{ \begin{pmatrix}\displaystyle \sum_i \lambda_i&\displaystyle \sum_i \lambda_ia_i&\displaystyle \sum_i \lambda_ib_i\\0&\displaystyle \sum_i \lambda_i&\displaystyle \sum_i \lambda_ic_i\\0&0&\displaystyle \sum_i \lambda_i\end{pmatrix} \}\\
			\end{split}
		\end{equation*}
		and morphisms : hom($L_{\mathbb{F}H_i}, L_{\mathbb{F}H_j}) = \{ \bar{f}_{ij} : L_{\mathbb{F}H_i} \rightarrow L_{\mathbb{F}H_j} \vert \bar{f}(\displaystyle \sum_i a_ig_i) = \displaystyle \sum_i a_if(g_i)$ where $f : H_i \rightarrow H_j$ is the group homomorphism  \} is the category $L_{\mathbb{F}(H(\mathbb{R}))}$  of Lie algebras of the group algebras of subgroups of $H(\mathbb{R})$.
	\end{exam}	
	\par Next we consider the category  whose objects are Plesken Lie algebras $\mathcal{L}(G)$ over a field $\mathbb{F}$ and morphisms are Plesken Lie algebra homomorphisms which we denote it by $ \mathcal{C}_{PLG}$.	
	\begin{exam}
		Consider the  Heisenberg group $H(\mathbb{R}) = \{ \begin{pmatrix}1&a&b\\0&1&c\\0&0&1\end{pmatrix} : a, b, c \in \mathbb{R} \}$ and its non-isomorphic subgroups $H_1, H_2, H_3$ and $H_4$ as given in Example.6. For each $A = \begin{pmatrix}1&a&b\\0&1&c\\0&0&1 \end{pmatrix} \in H_i$,
		\begin{equation*}
			\begin{split}
				\hat{A} &= A - A^{-1} = \begin{pmatrix}1&a&b\\0&1&c\\0&0&1 \end{pmatrix} - \begin{pmatrix}1&-a&ac-b\\0&1&-c\\0&0&1 \end{pmatrix} 
				=\begin{pmatrix}0&2a&2b-ac\\0&0&2c\\0&0&0 \end{pmatrix}
			\end{split}
		\end{equation*}
		Thus $\mathcal{L}(H_i) = \{ \displaystyle \sum_i \lambda_i \hat{A_i} : A_i \in H_i, \lambda_i \in \mathbb{F} \}$ and so
		\begin{equation*}
			\begin{split}
				\mathcal{L}(H_1) &= \{0\}, \ \indent \mathcal{L}(H_2)= \{ \begin{pmatrix}0&0&2\displaystyle \sum_i \lambda_ib_i\\0&0&0\\0&0&0 \end{pmatrix}\}\\
				\mathcal{L}(H_3)&= \{ \begin{pmatrix}0&0&2\displaystyle \sum_i \lambda_ib_i\\0&0&2\displaystyle \sum_i \lambda_ic_i\\0&0&0 \end{pmatrix} \}, 
				\mathcal{L}(H_4)= \{ \begin{pmatrix}0&2\displaystyle \sum_i \lambda_ia_i&\displaystyle \sum_i \lambda_i(2b_i - a_ic_i)\\0&0&2\displaystyle \sum_i \lambda_ic_i\\0&0&0 \end{pmatrix} \}\\						
			\end{split}
		\end{equation*}
		
		The category whose objects are $\mathcal{L}(H_i),i=1,2,3,4$ and morphisms are 
		$hom(\mathcal{L}(H_i), \mathcal{L}(H_j)) = \{ \hat{f} : \mathcal{L}(H_i) \rightarrow \mathcal{L}(H_j) \ \vert \ \hat{f}(\displaystyle \sum_i \lambda_i\hat{A_i}) = \displaystyle \sum_i \lambda_i\widehat{f(A_i)}$ where $f : H_i \rightarrow H_j$ is the group homomorphism \} is the category of Plesken Lie algebras of subgroups of $H(\mathbb{R})$.	
	\end{exam}	
	\begin{thm}
		If $L_{\mathbb{F}G}$ is the category of Lie algebras of group algebras and $\mathcal{C}_{PLG}$ is the category of Plesken Lie algebras, then 
		there exists a functor from $L_{\mathbb{F}G}$ to $\mathcal{C}_{PLG}$.
		
	\end{thm}
	\begin{proof}
		
		Define  $T : L_{\mathbb{F}G} \rightarrow \mathcal{C}_{PLG}$ such that 
		 $\nu T : L_{\mathbb{F}G_i} \rightarrow \mathcal{L}(G_i)$ is defined by 
		\begin{center}
			$\nu T(\displaystyle \sum_i a_ig_i) = \displaystyle \sum_{i}(a_i - a_i')(g_i - {g_i}^{-1})$
		\end{center}

		and on morphisms, for  a group homomorphism $f : G_i \rightarrow G_j$, $\bar{f} : L_{\mathbb{F}G_i} \rightarrow  L_{\mathbb{F}G_j}$ is given by
		\begin{center}
			$\bar{f}(\displaystyle \sum_i a_ig_i) = \displaystyle \sum_i a_if(g_i)$
		\end{center}
		 and $T\bar{f} : T(L_{\mathbb{F}G_i})  \rightarrow T(L_{\mathbb{F}G_j})$ is  defined by
		\begin{center}
			$T\bar{f} = \hat{f}$
		\end{center}
		and $\hat{f} : \mathcal{L}(G_i) \rightarrow \mathcal{L}(G_j)$ is given by 
		\begin{center}
			$\hat{f}(\displaystyle \sum_{i}(a_i- a_i')(g_i-{g_i}^{-1}))= \displaystyle \sum_{i}(a_i- a_i')(f(g_i)-f({g_i})^{-1})$
		\end{center}
		where $\hat{f}$ is a Plesken Lie algebra homomorphism ( Proposition 2).\\
		Let $\bar{f_1} : L_{\mathbb{F}G_i} \rightarrow L_{\mathbb{F}G_j}$ and $ \bar{f_2} : L_{\mathbb{F}G_j} \rightarrow L_{\mathbb{F}G_k}$ be two  homomorphisms between Lie algebras of group algebras which are induced from the group homomorphisms $f_1 : G_i \rightarrow G_j$ and $f_2 : G_j \rightarrow G_k$. Then $T(\bar{f_1}): T(L_{\mathbb{F}G_i}) \rightarrow T(L_{\mathbb{F}G_j})$ and $T(\bar{f_2}): T(L_{\mathbb{F}G_j}) \rightarrow T(L_{\mathbb{F}G_k})$ are Plesken Lie algebra homomorphisms. Then their composition $T(\bar{f_2}) \circ T(\bar{f_1}) : T(L_{\mathbb{F}G_i}) \rightarrow T(L_{\mathbb{F}G_k})$ is also a  Plesken Lie algebra homomorphism. Moreover,
	\begin{tiny}
		\begin{equation}
			\begin{split}
				(T(\bar{f_2}) \circ T(\bar{f_1}))(\displaystyle \sum_{i}(a_i- a_i')(g_i-{g_i}^{-1}))&=(\hat{f_2} \circ \hat{f_1})(\displaystyle \sum_{i}(a_i- a_i')(g_i-{g_i}^{-1})))\\
				&=\hat{f_2}(\displaystyle \sum_{i=1}(a_i- a_i')(f_1(g_i)-f_1({g_i}^{-1})))\\
				&= \displaystyle \sum_{i}(a_i- a_i')((f_2 \circ  f_1)(g_i)-((f_2 \circ f_1)(g_i))^{-1})
			\end{split}
		\end{equation}
	\end{tiny}
and
	\begin{tiny}
		\begin{equation}
			\begin{split}
				T(\bar{f_2} \circ \bar{f_1})(\displaystyle \sum_{i}(a_i- a_i')(g_i-{g_i}^{-1}))&= (\widehat{f_2 \circ f_1})(\displaystyle \sum_{i}(a_i- a_i')(g_i-{g_i}^{-1}))\\
				&=\displaystyle \sum_{i}(a_i- a_i')((f_2 \circ  f_1)(g_i)-((f_2 \circ f_1)(g_i))^{-1})
			\end{split}
		\end{equation} 
	\end{tiny}
		$(1)$ and $(2)$ shows that $	T(\bar{f_2} \circ \bar{f_1})= T(\bar{f_2}) \circ T(\bar{f_1})$.\\
		The identity Lie algebra homomorphism, $1_{L_{\mathbb{F}G_i}} : L_{\mathbb{F}G_i} \rightarrow L_{\mathbb{F}G_i}$ is  induced from the identity group homomorphism $1_{G_i}: G_i \rightarrow G_i$. Then $T(1_{L_{\mathbb{F}G_i}}) : T(L_{\mathbb{F}G_i}) \rightarrow T(L_{\mathbb{F}G_i})$ is a Plesken Lie algebra homomorphism and
		\begin{equation*}
			\begin{split}
				T(1_{L_{\mathbb{F}G_i}})(\displaystyle \sum_{i}(a_i- a_i')(g_i-{g_i}^{-1}))&= \displaystyle \sum_{i}(a_i- a_i')(1_{G_i}(g_i)-(1_{G_i}({g_i}))^{-1})\\
				&=\displaystyle \sum_{i}(a_i- a_i')(g_i-{g_i}^{-1})\\
				&= 1_{T(L_{\mathbb{F}G_i})}(\displaystyle \sum_{i}(a_i- a_i')(g_i-{g_i}^{-1}))
			\end{split}
		\end{equation*}
		That is, $T(1_{L_{\mathbb{F}G_i}}) = 1_{T(L_{\mathbb{F}G_i})}$, hence $T$ is a functor.
		
	\end{proof}
	\begin{cor}
		The functor $T : L_{\mathbb{F}G} \rightarrow \mathcal{C}_{PLG}$ in Theorem 1 is a full functor.
	\end{cor}
	\begin{proof}
		For the functor $T : L_{\mathbb{F}G} \rightarrow \mathcal{C}_{PLG}$ define a map $T_{L_{\mathbb{F}G_i}, L_{\mathbb{F}G_j}} : hom(L_{\mathbb{F}G_i}, L_{\mathbb{F}G_j}) \rightarrow hom(TL_{\mathbb{F}G_i}, TL_{\mathbb{F}G_j})$ by 
		\begin{center}
			 $T_{L_{\mathbb{F}G_i}, L_{\mathbb{F}G_j}}(\bar{f})= \hat{f}$
		\end{center}
		Now it is enough to prove that $T_{L_{\mathbb{F}G_i}, L_{\mathbb{F}G_j}}$
		is surjective. For  $\hat{f} \in hom(TL_{\mathbb{F}G_i}, TL_{\mathbb{F}G_j})$, a Plesken Lie algebra homomorphism, 
		define $\bar{f} : L_{\mathbb{F}G_i} \rightarrow L_{\mathbb{F}G_j}$ by
		\begin{center}
			$\bar{f}(\displaystyle \sum_i a_i g_i) = \displaystyle \sum_i a_i f(g_i)$
		\end{center}
		 is a homomorphism between Lie algebras of group algebras and
		$T\bar{f} : TL_{\mathbb{F}G_i} \rightarrow TL_{\mathbb{F}G_j}$ given by,
		\begin{center}
			$T\bar{f}(\displaystyle \sum_i a_i \hat{g_i}) = \displaystyle \sum_i a_i \widehat{f(g_i)} = \hat{f}(\displaystyle \sum_i a_i \hat{g_i})$
		\end{center}
		That is, $T\bar{f} = \hat{f}$. Hence $T$ is full.
	\end{proof}

	 However, it should be noted that the functor defined above need not be faithful as illustrated in the following example.
	
	\begin{exam}
		 Let  $K_4 = \{e, a, b, c\}$ be the Klein 4- group and $L_{\mathbb{F}K_4}$ Lie algebra  of the group algebra $\mathbb{F}K_4$. Consider the identity map $1_{L_{\mathbb{F}K_4}}$  and $\bar{f} : L_{\mathbb{F}K_4} \rightarrow L_{\mathbb{F}K_4}$ given by
		 \begin{center}
		 	$\bar{f}(\displaystyle \sum_i a_ig_i) = \displaystyle \sum_i a_if(g_i)$
		 \end{center}
	 where $g_i \in K_4$ and $f$ is the trivial homomorphism. Clearly both
		  $ T\bar{f}$ and $T1_{L_{\mathbb{F}K_4}}$ are zeroes for the functor $T : L_{\mathbb{F}G} \rightarrow \mathcal{C}_{PLG}$.
		
	\end{exam}

\end{document}